\documentclass[a4paper,12pt,oneside]{article}

\usepackage{comment} 

\usepackage{url}
\usepackage{cite}

\usepackage{multirow}
\usepackage{makecell}
\usepackage{subcaption} 

\usepackage{authblk}
\usepackage{blindtext}
\usepackage{mathtools}
\usepackage{amssymb, amscd, amsmath,amsthm,amsfonts,enumerate}

\usepackage{graphicx}
\usepackage{epsfig}
\usepackage{epstopdf}
\usepackage{color}
\usepackage[utf8]{inputenc}
\usepackage{blindtext}
\usepackage{nomencl}
\usepackage{fullpage}
\usepackage{bm}
\usepackage{cleveref}
\usepackage{mathtools}
\usepackage{array}
\vspace{1cm}
\usepackage{float}
\usepackage{multirow}
\usepackage{listings}
\usepackage{subfig}
\usepackage{centernot}
\usepackage[format=plain,labelfont=it,textfont=it]{caption}
\usepackage{titlesec}

\usepackage{breqn}
\usepackage{mathrsfs}
\graphicspath{{figures/}}
\makenomenclature
\newcommand{\doublespace}

\newcommand*\oline[1]{%
	\vbox{%
		\hrule height 0.5pt
		\kern0.25ex
		\hbox{%
			\kern-0.1em
			\ifmmode#1\else\ensuremath{#1}\fi
			\kern-0.1em
		}
	}
}
\newtheorem{theorem}{Theorem}[section]

\newtheorem{lemma}[theorem]{Lemma}

\newtheorem{corollary}[theorem]{Corollary}
\newtheorem{claim}{Claim}

\usepackage{dsfont}

\usepackage{chemfig}
\usepackage[english]{babel}

\begin{document}	
	
\baselineskip=0.30in

\begin{center}
{\Large \textbf{On Hyperbolic Sombor index of graphs}
} 
\\
		
\vspace{8mm}
        
{\bf Kinkar Chandra Das$^{1}$, Sultan Ahmad$^{2,*}$}

\vspace{9mm}

\baselineskip=0.20in

$^1${\it Department of Mathematics, Sungkyunkwan University,
Suwon 16419, Republic of Korea\/} \\
{\rm E-mail:} {\tt kinkardas2003@gmail.com}\\[2mm]

$^2${\it Department of Mathematics, School of Natural Sciences, National University of Sciences and Technology, H-12, Islamabad 44000, Pakistan\/} \\
{\rm E-mail:} {\tt raosultan58@gmail.com}

\vspace{4mm}
		
\end{center}
	
\vspace{5mm}
	
\baselineskip=0.20in

\begin{abstract}
The Hyperbolic Sombor index $HSO(G)$ of a graph $G$ is defined as
$$
HSO(G) = \sum_{v_i v_j \in E(G)} \frac{\sqrt{d_i^2 + d_j^2}}{\min\{d_i, d_j\}},
$$
where $d_i$ and $d_j$ denote the degrees of the vertices $v_i$ and $v_j$, respectively. This index was recently introduced by Barman et al. {\bf[Geometric approach to degree-based topological index: Hyperbolic Sombor index, MATCH Commun. Math. Comput. Chem. 95 (2026) 63–94]}, who explored some of its mathematical properties and applications. However, their work contains several inaccuracies that require correction. In this paper, we first identify and rectify the errors found in the earlier study. We then extend the investigation by establishing new mathematical results for the Hyperbolic Sombor index across various classes of graphs, including trees, unicyclic graphs, and bicyclic graphs. In addition, we derive some lower and upper bounds for $HSO(G)$ in terms of the number of edges, maximum degree and minimum degree, and we charcterize the graphs that attains these bounds. Finally, we conclude the paper by outlining potential directions for future research in this emerging area.


\end{abstract}
  
\baselineskip=0.25in
  
\section{Introduction}

All graphs considered in this paper are finite, simple, and connected. For a graph $G$, we denote its vertex set and edge set by $V(G)$ and $E(G)$, respectively. The order and size of $G$ are given by $|V(G)|=n$ and $|E(G)|=m$, respectively. The edge set $E(G)$ contains $e=v_iv_j$ if and only if the vertices $v_i$ and $v_j$ are adjacent. The graph modifications of edge removal and insertion are expressed as $G - \{v_i v_j\}$ and $G + \{v_i v_j\}$, respectively. The degree of a vertex $v_i \in V(G)$, written $d_i$, is the number of neighbors of $v_i$ in $G$, and the set of neighbors of $v_i$ is denoted by $N_G(v_i)$. The maximum and minimum degrees of $G$ are denoted by $\Delta(G)$ and $\delta(G)$, respectively. A graph $G$ is said to be regular if $\Delta(G)=\delta(G)$.
A vertex of a graph $G$ with degree one is called a pendent vertex, and the edge incident to it is called a pendent edge.
A graph with $m = n+c-1$ edges is referred to as a $c$-cyclic graph. In particular, for $c=0,1,2$, such graphs are called trees, unicyclic graphs, and bicyclic graphs, respectively. 

We use the standard notations $S_n$, $P_n$, $C_n$, and $K_n$ to denote the star, path, cycle and complete graphs of order $n$, respectively. The degree sequence of an $n$-vertex graph $G$ is the sequence $(d_1, d_2, \ldots, d_n)$, where $d_i$ denotes the degree of the $i$-th vertex of $G$, arranged in non-increasing order, that is, $d_1 \geq d_2 \geq \cdots \geq d_n$. Whenever the graph $G$ is clear from the context, we omit the $(G)$ part in the notation. For undefined terminology and notation, we refer to \cite{B1}.

Chemical graph theory, a branch of mathematical chemistry, employs graph-theoretical concepts to model and analyze molecular structures. In this framework, molecules are represented 
as graphs in which atoms correspond to vertices and chemical bonds to edges, providing a rigorous mathematical basis for investigating molecular properties. Molecular descriptors, commonly referred to as topological indices when defined on molecular graphs, are fundamental tools for virtual 
screening and for predicting physicochemical properties of molecules \cite{Basak,Desmecht}.  

Among these, degree-based topological indices play a central role due to their computational efficiency and predictive 
power, making them indispensable in quantitative structure–property relationship (QSPR) studies. Such indices are generally defined as a sum, over all edges of a graph, of quantities that depend on the degrees of the end-vertices.  

A significant development in this area arose from Gutman’s introduction of a geometric 
interpretation for the Sombor index \cite{gutman2021geometric}, which has inspired extensive research in both mathematics \cite{SAhmadKCDas,XY1,XY2,XY3,SD1,chen2022extremal,PM,JRJMR,liu,maitreyi2023minimum,shetty,wang2022relations} and chemistry \cite{redvzepovic2021chemical}, as well as comprehensive surveys 
\cite{liu2022sombor,BAR}. This geometric viewpoint has since led to a growing family of related indices, including 
the elliptic Sombor \cite{gutman2024geometric,SDF}, Euler–Sombor \cite{DJ1,IGutman,tang2024euler}, 
Hyperbolic Sombor \cite{BD1}, diminished Sombor \cite{K1,MovahediGutman}, and augmented Sombor \cite{KCD1} indices. In this paper, we focus on one such Hyperbolic geometric interpretation–based index, namely the Hyperbolic Sombor index ($HSO$). For a graph $G$, it is defined as
\begin{align*}
HSO(G) = \sum_{v_iv_j \in E(G)} \frac{\sqrt{d_i^{2}+d_j^{2}}}{\min\{d_i,\,d_j\}}.
\end{align*}
In \cite{BD1}, Barman and Das established several fundamental results concerning the Hyperbolic Sombor index $HSO(G)$. They derived lower bounds in terms of the size of a graph, obtained inequalities involving the Sombor index together with the maximum and minimum degree, and provided bounds using the first Zagreb index. Furthermore, they proved that among all graphs, the star graph maximizes and the cycle graph minimizes $HSO(G)$, while within the class of trees, the star graph maximizes and the path graph minimizes this index. They also examined its predictive power, structure sensitivity, and degeneracy in the context of alkane isomers. In a subsequent study \cite{BD2}, Barman and Das investigated the chemical significance of $HSO(G)$ through curvilinear regression with physicochemical properties of benzenoid hydrocarbons. 

The primary motivation of the present study arises from \cite{BD1}, as several of their results contain flaws. In this paper, we both highlight these flaws in the original results and establish new mathematical findings for the Hyperbolic Sombor index across various classes of graphs, including trees, unicyclic graphs, and bicyclic graphs. In addition, we derive some lower and upper bounds for $HSO(G)$ in terms of the number of edges, maximum degree and minimum degree, and we charcterize the graphs that attains these bounds.

The remainder of this paper is organized as follows. In Sect.~\ref{section3}, we point out flaws in the results of \cite{BD1}. In Sect.~\ref{section2}, we present new results for the Hyperbolic Sombor index for unicyclic and bicyclic graphs. In Sect.~\ref{section4}, we derive some lower and upper bounds for $HSO(G)$ in terms of the number of edges, maximum degree and minimum degree, and we characterize the graphs that attains these bounds. Finally, in Sect.~\ref{section5}, we conclude the paper.

\section{Flaws in the results of \cite{BD1}}\label{section3}
In this section, we point out the flaws in the results of \cite{BD1}. We begin with the following key observation. \\
{\bf Remark 1.} In the statements of the Theorems 1, 2 and 3 in \cite{BD1}, the authors mentioned that ``the equality holds if and only if $G$ is a complete graph". But these statements are not correct. 

\vspace*{2mm}

\noindent
{\bf Comment on Theorems 1 and 3 \cite{BD1}:} The inequalities presented in the theorems are correct; however, the statements regarding equality are not accurate. The correct formulation is: `` Equality holds if and only if $G$ is a regular graph". Both proofs rely on the following inequality:

\vspace*{2mm}

For any edge $v_iv_j\in E(G)$, we have $(d_i-d_j)^2\geq 0$ with equality if and only if $d_i=d_j$. Since $G$ is connected, equality in the statements of the theorems holds if and only if $d_1=d_2=\cdots=d_n$, that is, if and only if $G$ is a regular graph. 

\vspace*{2mm}

\noindent
{\bf Comment on Theorem 2 \cite{BD1}:} The inequality stated in Theorem 2 are correct, but the condition for equality is not. To ensure correctness and completeness, we provide a revised proof of the theorem. First, we define a graph $H$ as follows:

\vspace*{2mm}

Let $\Gamma$ be the class of connected graphs $H=(V,E)$ of order $n$, i.e., $|V(H)| = n$, with minimum degree $\delta$, such that the vertex set $V(H)$ is partitioned into two disjoint subsets $U$ and $W$, satisfying the following conditions:
\begin{itemize}
\item $V(H)=U \cup W$, and $U \cap W = \emptyset$,
\item there are no edges within $U$ $($i.e., $H[U]$ is an empty graph$)$,
\item edges may or may not exist within $W$,
\item for every vertex $v_i \in W$, the degree $d_i = \delta$.
\end{itemize}
Four graphs $H_1,\,H_2,\,H_3$ and $H_4$ have been shown in Fig.~\ref{fig5}. In particular, $H_1$ is a graph of order $7$, $H_2$ and $H_3$ are graphs of order $8$, while $H_4$ is a graph of order $9$. One can easily see that $H_i\in \Gamma$ for $1\leq i\leq 4$.
 \begin{figure}[ht!]
 \begin{center}
 \includegraphics[width=12.9cm]{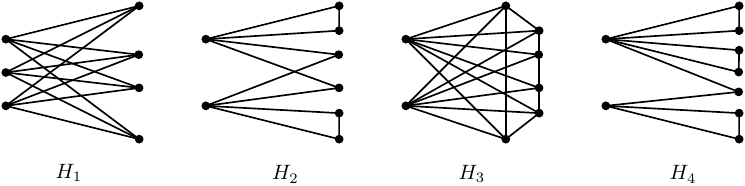}
\caption{\small \sl Four graphs $H_1,\,H_2,\,H_3$, and $H_4$.}
  \label{fig5}
 \end{center}
 \end{figure}

We now provide a revision of Theorem 2 from \cite{BD1}, where we establish a relationship between the Sombor index and the Hyperbolic Sombor index of any connected graph $G$.
\begin{theorem} Let $G$ be a simple connected graph with maximum degree $\Delta$, minimum degree $\delta$, and Sombor index $SO(G)$. Then 
$$\frac{1}{\Delta} \cdot SO(G) \leq HSO(G) \leq \frac{1}{\delta} \cdot SO(G).$$
Moreover, the left equality holds if and only if $G$ is a regular graph, and the right equality holds if and only if $G$ is a regular graph or $G\in \Gamma$.
\end{theorem}

\begin{proof} Since $\delta\leq d_i\leq \Delta\,(1\leq i\leq n)$, we have $\frac{1}{\Delta}\leq \frac{1}{d_i}\leq \frac{1}{\delta}$.

\vspace*{2mm}

\noindent
{\bf Upper Bound:} Using the above, from the definition of the Hyperbolic Sombor index, we obtain
\begin{align*}
HSO(G)&=\sum\limits_{v_iv_j\in E(G)}\,\frac{\sqrt{d^{2}_i+d^{2}_j}}{\min\{d_i,\,d_j\}}\leq \sum\limits_{v_iv_j\in E(G)}\,\frac{\sqrt{d^{2}_i+d^{2}_j}}{\delta}=\frac{1}{\delta}\cdot SO(G).
\end{align*}
The above equality holds if and only if $\delta=\min\{d_i,\,d_j\}$ for any edge $v_iv_j\in E(G)$. 

\vspace*{2mm}

Suppose that the above equality holds. Then $\delta=\min\{d_i,\,d_j\}$ for any edge $v_iv_j\in E(G)$. Since $G$ is connected, we can partition the vertex set as $V(G)=U\cup W$ with $U\cap W=\emptyset$, where $W$ consists of all vertices of degree $\delta$, while $U$ contains the remaining vertices, each having degree strictly greater than $\delta$, and there are no edges between vertices within $U$ ( that is, the induced subgraph $G[U]$ is empty). If $|U|=0$, then all vertices have the same degree $\delta$, and thus $G$ is a regular graph. Otherwise, $|U|>0$. Then $G\cong H$, where $H\in \Gamma$.

\vspace*{2mm}

Conversely, let $G$ be an $r$-regular graph. Then $\delta=r$, $HSO(G)=\sqrt{2}\cdot |E(G)|$, and $SO(G)=\sqrt{2}\,r\cdot |E(G)|$. Thus we have $HSO(G)=\frac{1}{\delta}\cdot SO(G)$.

\vspace*{2mm}

Let $G\cong H$, where $H\in \Gamma$. From the definition, for any edge $v_iv_j\in E(G)$ with $d_i\geq d_j=\delta$, that is, $\displaystyle{\frac{\sqrt{d^{2}_i+d^{2}_j}}{\min\{d_i,\,d_j\}}=\frac{\sqrt{d^{2}_i+d^{2}_j}}{\delta}}$. Hence $$HSO(G)=\sum\limits_{v_iv_j\in E(G)}\,\frac{\sqrt{d^{2}_i+d^{2}_j}}{\min\{d_i,\,d_j\}}=\sum\limits_{v_iv_j\in E(G)}\,\frac{\sqrt{d^{2}_i+d^{2}_j}}{\delta}=\frac{1}{\delta}\cdot SO(G).$$

\vspace*{2mm}

\noindent
{\bf Lower Bound:} Similarly to the upper bound, we obtain:
\begin{align*}
HSO(G)&=\sum\limits_{v_iv_j\in E(G)}\,\frac{\sqrt{d^{2}_i+d^{2}_j}}{\min\{d_i,\,d_j\}}\geq \sum\limits_{v_iv_j\in E(G)}\,\frac{\sqrt{d^{2}_i+d^{2}_j}}{\Delta}=\frac{1}{\Delta}\cdot SO(G).
\end{align*}
The above equality holds if and only if $\Delta=\min\{d_i,\,d_j\}$ for any edge $v_iv_j\in E(G)$, that is, if and only if $\Delta=\min\{d_i,\,d_j\}\leq \Delta$ for any edge $v_iv_j\in E(G)$, that is, if and only if $d_i=d_j=\Delta$ for any edge $v_iv_j\in E(G)$, that is, if and only if $G$ is a regular graph as $G$ is connected. 
\end{proof}

 \begin{figure}[ht!]
 \begin{center}
 \includegraphics[height=3.9cm]{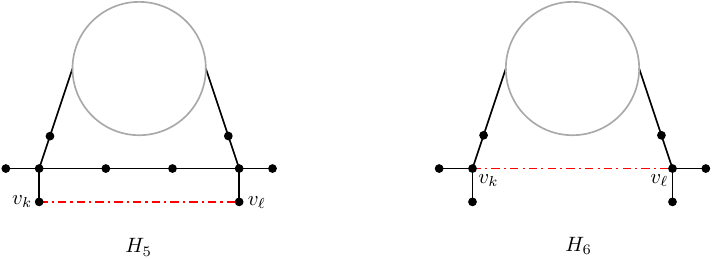}
 \caption{\small \sl Two graphs $H_5$ and $H_6$.}
  \label{fig1}
 \end{center}
 \end{figure}

\noindent
{\bf Remark 2.}
In the proof of Theorem 4 in \cite{BD1}, the authors stated that ``it is obvious that the value of $HSO(G)$ increases when we add edges to the graph $G$". Furthermore, they also mentioned that in the proof of Theorem 5 ``it is obvious that the value of $HSO(G)$ decreases when we remove edges from the graph $G$". However, these statements are not correct.

To illustrate this, consider the two graphs $H_{5}$ and $H_{6}$ shown in Fig.~\ref{fig1}. A direct calculation gives  
$$HSO(H_5)-HSO(H_5+v_kv_{\ell}) = 2\sqrt{17}-2\sqrt{5}-\sqrt{2} > 2.35 > 0,$$  
which shows that $HSO(H_5) > HSO(H_5+v_kv_{\ell})$. Similarly,  
$$HSO(H_6+v_kv_{\ell})-HSO(H_6) = 2\sqrt{5}+4\sqrt{17}+\sqrt{2}-2\sqrt{\tfrac{13}{4}}-4\sqrt{10} > 6.12 > 0,$$  
implying that \(HSO(H_6+v_kv_{\ell}) > HSO(H_6)\).  

\vspace*{2mm}

\noindent
{\bf Remark 3.} In Theorem 5 of \cite{BD1}, the proof of the inequality $HSO(T)\leq HSO(S_n)$ is not entirely rigorous. Although the argument is conceptually simple, the authors omitted a necessary justification: for any edge $v_iv_j\in E(T)$, $$\frac{\sqrt{d^{2}_i+d^{2}_j}}{\min\{d_i,\,d_j\}}\leq \sqrt{(n-1)^2+1}.$$ 
Without this bound, the inequality lacks completeness. 

\vspace*{2mm}

However, the proof of the inequality $HSO(T)\geq HSO(P_n)$, as presented in the same theorem, is incorrect. The authors attempted to establish this result using the principle of mathematical induction, but the method was not applied properly. Specifically, to prove $HSO(T_{k+1})\geq HSO(P_{k+1})$, one must use the induction hypothesis $HSO(T_{k})\geq HSO(P_{k})$, where $T_{k+1}$ and $T_k$ are any trees of order $k+1$ and $k$, respectively. Instead, the authors considered a very special kind of tree for $T_{k+1}$, rather than an arbitrary tree of order $k+1$, which invalidates the generality of their argument. Therefore, the proof method is flawed. This issue is similar to what occurs in Theorem 4 of \cite{BD1}, where the proof contains similar flaws. 

\vspace*{2mm}

Based on {\bf Remarks 2 and 3}, it is evident that the proofs of Theorems 4 and 5 in \cite{BD1} are incorrect. The concepts are wrong. We now revise the proof of the Theorem 5 in \cite{BD1}. For this first we prove the following result. 
\begin{lemma}\label{k1} Let $G$ be a graph with maximum degree $\Delta$. Then 

\vspace{1mm}

\noindent
$(i)$ For any pendent edge $v_iv_j\in E(G)$, we have
  $$\sqrt{5}\leq \frac{\sqrt{d^{2}_i+d^{2}_j}}{\min\{d_i,\,d_j\}}\leq \sqrt{\Delta^2+1}$$
with equality on the left if and only if $d_i=2>1=d_j$, and equality on the right if and only if $d_i=\Delta>1=d_j$.

\vspace{2mm}

\noindent
$(ii)$ For any non-pendent edge $v_iv_j\in E(G)$, we have
  $$\sqrt{2}\leq \frac{\sqrt{d^{2}_i+d^{2}_j}}{\min\{d_i,\,d_j\}}\leq \frac{1}{2}\,\sqrt{\Delta^2+4}$$
with equality on the left if and only if $d_i=d_j$, and equality on the right if and only if $d_i=\Delta>2=d_j$.
\end{lemma}

\begin{proof} Let $v_iv_j$ be any edge in $G$ with $d_i\geq d_j$. Then  
\begin{align}
\frac{\sqrt{d^{2}_i+d^{2}_j}}{\min\{d_i,\,d_j\}}=\frac{\sqrt{d^{2}_i+d^{2}_j}}{d_j}=\sqrt{\frac{d^2_i}{d^2_j}+1}.\label{e1}
\end{align}

\vspace{1mm}
\noindent
$(i)$ For any pendent edge $v_iv_j\in E(G)\,(1=d_j<d_i\leq \Delta)$, we have
   $$\sqrt{5}\leq \sqrt{\frac{d^2_i}{d^2_j}+1}=\sqrt{d^2_i+1}\leq \sqrt{\Delta^2+1},$$
where $\Delta$ is the maximum degree in $G$. Using the above result in (\ref{e1}), we obtain
  $$\sqrt{5}\leq\frac{\sqrt{d^{2}_i+d^{2}_j}}{\min\{d_i,\,d_j\}}\leq \sqrt{\Delta^2+1}.$$

\vspace{2mm}
\noindent
$(ii)$ For any non-pendent edge $v_iv_j\in E(G)\,(2\leq d_j\leq d_i\leq \Delta)$, we have
   $$\sqrt{2}\leq \sqrt{\frac{d^2_i}{d^2_j}+1}\leq \frac{1}{2}\,\sqrt{\Delta^2+4},$$
where $\Delta$ is the maximum degree in $G$. Using the above result in (\ref{e1}), we obtain
  $$\sqrt{2}\leq\frac{\sqrt{d^{2}_i+d^{2}_j}}{\min\{d_i,\,d_j\}}\leq \frac{1}{2}\,\sqrt{\Delta^2+4}.$$
This completes the proof of the result.
\end{proof}

\begin{corollary}\label{tk1} Let $G$ be a graph of order $n$. Then

\vspace{1mm}

\noindent
$(i)$ For any pendent edge $v_iv_j\in E(G)$, we have
  $$\sqrt{5}\leq \frac{\sqrt{d^{2}_i+d^{2}_j}}{\min\{d_i,\,d_j\}}\leq \sqrt{(n-1)^2+1}$$
with equality on the left if and only if $d_i=2>1=d_j$, and equality on the right if and only if $d_i=n-1>1=d_j$.

\vspace{2mm}

\noindent
$(ii)$ For any non-pendent edge $v_iv_j\in E(G)$, we have
  $$\sqrt{2}\leq \frac{\sqrt{d^{2}_i+d^{2}_j}}{\min\{d_i,\,d_j\}}\leq \frac{1}{2}\,\sqrt{(n-1)^2+4}$$
with equality on the left if and only if $d_i=d_j$, and equality on the right if and only if $d_i=n-1>2=d_j$.
\end{corollary}

\begin{theorem} \label{kh1} Let $T$ be a tree of order $n$. Then 
$$2\,\sqrt{5}+(n-3)\,\sqrt{2}\leq HSO(T)\leq (n-1)\,\sqrt{n^2-2n+2}$$ 
with left equality if and only if $T\cong P_n$ and the right equality if and only if $T\cong S_n$. 
\end{theorem}

\begin{proof} {\bf Upper Bound:} Let $v_iv_j$ be any edge in $T$ with $d_i\geq d_j$. Then by Corollary \ref{tk1}, we obtain  
   $$\frac{\sqrt{d^{2}_i+d^{2}_j}}{\min\{d_i,\,d_j\}}\leq \sqrt{(n-1)^2+1}.$$
Thus we have
\begin{align*}
HSO(T)&=\sum\limits_{v_iv_j\in E(T)}\,\frac{\sqrt{d^{2}_i+d^{2}_j}}{\min\{d_i,\,d_j\}}\leq (n-1)\,\sqrt{n^2-2n+2}.
\end{align*}
Moreover, the above equality holds if and only if $d_i=n-1>1=d_j$ for any edge $v_iv_j\in E(T)$, that is, if and only if $T\cong S_n$. 

\vspace*{3mm}

\noindent
{\bf Lower Bound:} Let $p$ be the number of pendent vertices in $T$. Then $p\geq 2$. By Corollary \ref{tk1}, we obtain
\begin{align*}
HSO(T)&=\sum\limits_{v_iv_j\in E(T),\atop d_i\geq d_j=1}\,\frac{\sqrt{d^{2}_i+d^{2}_j}}{\min\{d_i,\,d_j\}}+\sum\limits_{v_iv_j\in E(T),\atop d_i\geq d_j\geq 2}\,\frac{\sqrt{d^{2}_i+d^{2}_j}}{\min\{d_i,\,d_j\}}\\[3mm]
&\geq p\,\sqrt{5}+(n-p-1)\,\sqrt{2}\\[2mm]
&=(n-1)\,\sqrt{2}+p\,(\sqrt{5}-\sqrt{2})\\[2mm]
&\geq (n-1)\,\sqrt{2}+2\,(\sqrt{5}-\sqrt{2})=2\,\sqrt{5}+(n-3)\,\sqrt{2}.
\end{align*}
Moreover, the equality holds if and only if $p=2$ with $d_i=d_j$ for any non-pendent edge $v_iv_j\in E(T)$, and $d_i=2>1=d_j$ for any pendent edge $v_iv_j\in E(T)$, that is, if and only if $T\cong P_n$.
\end{proof}

In the next result, we revisit Theorem 4 from \cite{BD1} and offer a revised proof for one of its parts.
\begin{theorem} Let $G$ be a connected graph of order $n$. Then $HSO(G)\geq \sqrt{2}\,n$ with equality if and only if $G\cong C_n$.
\end{theorem}

\begin{proof} Let $m$ be the number of edges in $G$. Since $G$ is connected, we have $m\geq n-1$. If $m=n-1$, then $G$ is a tree $T$. By Theorem \ref{kh1}, we have 
  $$HSO(T)\geq 2\,\sqrt{5}+(n-3)\,\sqrt{2}$$ 
with equality if and only if $T\cong P_n$. Using this result, we obtain
 $$HSO(G)=HSO(T)=\sum\limits_{v_iv_j\in E(T)}\,\frac{\sqrt{d^{2}_i+d^{2}_j}}{\min\{d_i,\,d_j\}}\geq 2\,\sqrt{5}+(n-3)\,\sqrt{2}>\sqrt{2}\,n.$$
The inequality strictly holds. Otherwise, $m\geq n$. By Lemma \ref{k1}, we obtain 
\begin{align*}
HSO(G)=\sum\limits_{v_iv_j\in E(T)}\,\frac{\sqrt{d^{2}_i+d^{2}_j}}{\min\{d_i,\,d_j\}}\geq \sqrt{2}\,m\geq \sqrt{2}\,n.
\end{align*}
Since $G$ is connected, the above equality holds if and only if $m=n$ and $d_i=d_j$ for any edge $v_iv_j\in E(G)$, that is, if and only if $G$ is a regular graph with $m=n$, that is, if and only if $G\cong C_n$.
\end{proof}

\noindent
{\bf Remark 4.} The other part of Theorem 4 of \cite{BD1}, $HSO(G)\leq HSO(S_n)$ is still open.

\section{Extremal Results on Unicyclic and Bicyclic Graphs with Respect to the Hyperbolic Sombor Index}\label{section2}
This section deals with two classes of graphs: unicyclic graphs, and bicyclic graphs. In order to establish our main results, we first present the following auxiliary result. 
\begin{lemma} \label{k2} Let 
  $$f(x)=x\,\sqrt{(x+2)^2+1}+(n-x-3)\,\sqrt{(n-x-1)^2+1},~1\leq x\leq \Big\lfloor\frac{n-3}{2}\Big\rfloor.$$
Then $f(x)$ is a decreasing function on $1\leq x\leq \Big\lfloor\frac{n-3}{2}\Big\rfloor$.
\end{lemma}

\begin{proof} Since
$$f(x)=x\,\sqrt{(x+2)^2+1}+(n-x-3)\,\sqrt{(n-x-1)^2+1},$$
we have 
\begin{align*}
&f'(x)\\
&=\frac{(x+2)^2+1+x\,(x+2)}{\sqrt{(x+2)^2+1}}-\frac{(n-x-1)^2+1+(n-x-1)\,(n-x-3)}{\sqrt{(n-x-1)^2+1}}.    
\end{align*}
Since $1\leq x\leq \Big\lfloor\frac{n-3}{2}\Big\rfloor$, we have $n-x-3\geq x$ and $n-x-1\geq x+2$. We now prove the following claim.
\begin{claim}\label{c1} 
\begin{align*}
 &\Big[(x+2)^2+1\Big]\,\Big[2\,(n-x-1)\,(n-x-2)+1\Big]\\[3mm]
 &~~~~~~~~~~~~~~~~~~~~~~\geq\Big[(n-x-1)^2+1\Big]\,\Big[2\,(x+2)\,(x+1)+1\Big].   
\end{align*}
\end{claim}

\vspace*{3mm}

\noindent
{\bf Proof of Claim \ref{c1}.} We have to prove that
\begin{align*}
&2\,(x+2)^2\,(n-x-1)\,(n-x-2)+2\,(n-x-1)\,(n-x-2)+(x+2)^2\\[2mm]
 &~~~~~~\geq 2\,(x+2)\,(x+1)\,(n-x-1)^2+2\,(x+2)\,(x+1)+(n-x-1)^2,
\end{align*}
that is,
\begin{align*}
&2\,(x+2)\,(n-x-1)\,\Big[(x+2)\,(n-x-2)-(x+1)\,(n-x-1)\Big]\\[2mm]
&~~~~~~~~~~~~~~~~~~~~~~~~~~~~+(n-x-1)\,(n-x-3)-(x+2)x\geq 0,
\end{align*}
that is,
\begin{align*}
&2\,(x+2)\,(n-x-1)\,(n-2x-3)+(n-x-1)\,(n-x-3)-(x+2)x\\[2mm]&
~~~~~~~~~~~~~~~~~~~~~~~~~~~~~~~~~~~~~~~~~~~~~~~~~~~\geq 0,
\end{align*}
which is true as $n-x-3\geq x$ and $n-x-1\geq x+2$. Hence
\begin{align*}
&\Big[(x+2)^2+1\Big]\,\Big[2\,(n-x-1)\,(n-x-2)+1\Big]\\[2mm]&~~~~~~~~~~~~~~~~~~~~~~~~~~~~~\geq\Big[(n-x-1)^2+1\Big]\,\Big[2\,(x+2)\,(x+1)+1\Big].    
\end{align*}
This proves the {\bf Claim \ref{c1}}.

\vspace*{2mm}
Again since $n-x-2\geq x+1$ and $n-x-1\geq x+2$, we obtain
 $$2\,(n-x-1)\,(n-x-2)+1\geq 2\,(x+2)\,(x+1)+1.$$
From {\bf Claim \ref{c1}} with the above result, we obtain
\begin{align*}
&\Big[(x+2)^2+1\Big]\,\Big[2\,(n-x-1)\,(n-x-2)+1\Big]^2
\\[2mm]
&~~~~~~~~~~~~~~~~~~~~~~~~~~\geq\Big[(n-x-1)^2+1\Big]\,\Big[2\,(x+2)\,(x+1)+1\Big]^2,   
\end{align*}
that is,
\begin{align*}
&\sqrt{(x+2)^2+1}\,\Big[2\,(n-x-1)\,(n-x-2)+1\Big]\\[2mm]
&~~~~~~~~~~~~~~~\geq\sqrt{(n-x-1)^2+1}\,\Big[2\,(x+2)\,(x+1)+1\Big],    
\end{align*}
that is,
$$\frac{(x+2)^2+1+x\,(x+2)}{\sqrt{(x+2)^2+1}}\leq\frac{(n-x-1)^2+1+(n-x-1)\,(n-x-3)}{\sqrt{(n-x-1)^2+1}}.$$
Using the above result, we conclude that $f'(x)\leq 0$. This proves the result.
\end{proof}

 \begin{figure}[ht!]
 \begin{center}
 \includegraphics[height=4cm]{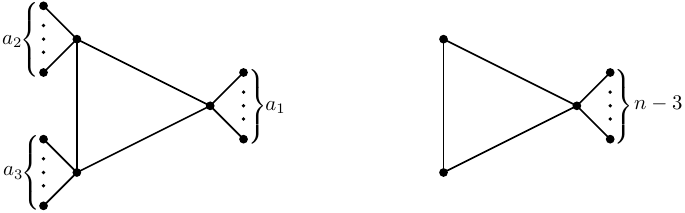}\\
$S(a_1,a_2,a_3)$~~~~~~~~~~~~~~~~~~~~~~~~~~~~~~~~~~~~~~~~~~~~$S(n-3,0,0)\cong S'_n$
 \caption{\small \sl Two graphs $S(a_1,a_2,a_3)$ and $S(n-3,0,0)\cong S'_n$.}
  \label{fig2}
 \end{center}
 \end{figure}
 
Let $S(a_1,a_2,a_3)$ be a unicyclic graph of order $n$ with cycle length $3$, where $n-3$ pendent vertices are attached to the vertices of the cycle $C_3$. Specifically, $a_1, a_2,$ and $a_3$ denote the numbers of pendent vertices attached to the three vertices of $C_3$, satisfying $a_1+a_2+a_3=n-3$ and $a_1 \geq a_2 \geq a_3 \geq 0$. In particular, the graph $S'_n$ is a special case of $S(a_1,a_2,a_3)$, obtained by setting $(a_1,a_2,a_3)=(n-3,0,0)$. For an illustration, see Fig.~\ref{fig2}. Now,
\begin{align}
HSO(S(a_1,a_2,a_3))&=a_1\,\sqrt{(a_1+2)^2+1}+a_2\,\sqrt{(a_2+2)^2+1}\nonumber\\[3mm]
&~~+a_3\,\sqrt{(a_3+2)^2+1}+\sqrt{\left(\frac{a_1+2}{a_2+2}\right)^2+1}\nonumber\\[3mm]
&~~+\sqrt{\left(\frac{a_2+2}{a_3+2}\right)^2+1}+\sqrt{\left(\frac{a_1+2}{a_3+2}\right)^2+1}.\label{r0}
\end{align}

We now establish lower and upper bounds for $HSO(G)$ in unicyclic graphs in terms of their order $n$, and characterize the corresponding extremal graphs.
\begin{theorem} Let $G$ be a unicyclic graph of order $n$. Then 
\begin{align}
\sqrt{2}\,n\leq HSO(G)\leq (n-3)\,\sqrt{n^2-2n+2}+\sqrt{n^2-2n+5}+\sqrt{2}\label{e2}
\end{align}
with equality on the left if and only if $G\cong C_n$ and equality on the right if and only if $G\cong S'_n$. 
\end{theorem}

\begin{proof} {\bf Upper Bound:} Let $\Delta$ and $p$ be the maximum degree and the number of pendent vertices in $G$. By Lemma \ref{k1}, for any pendent edge $v_iv_j\in E(G)$ with $d_i>d_j=1$, we have
\begin{align}
\frac{\sqrt{d^{2}_i+d^{2}_j}}{\min\{d_i,\,d_j\}}\leq \sqrt{\Delta^2+1}.\label{e3}
\end{align}
By Lemma \ref{k1}, for any non-pendent edge $v_iv_j\in E(G)$ with $d_i\geq d_j$, we have
\begin{align}
\frac{\sqrt{d^{2}_i+d^{2}_j}}{\min\{d_i,\,d_j\}}\leq \frac{1}{2}\,\sqrt{\Delta^2+4}.\label{e4}
\end{align}
Since $G$ is unicyclic, we have $p\leq n-3$. We consider the following two cases:

\vspace*{2mm}

\noindent
${\bf Case\,1.}$ $p=n-3$. In this case $G\cong S(a_1,a_2,a_3)$, where $a_1+a_2+a_3=n-3$ and $a_1\geq a_2\geq a_3\geq 0$. If $a_2=0$, then $a_2=a_3=0$ and $a_1=n-3$, that is, $G\cong S(n-3,0,0)\cong S'_n$ with
  $$HSO(G)=(n-3)\,\sqrt{n^2-2n+2}+\sqrt{n^2-2n+5}+\sqrt{2}$$ 
and hence the equality holds. Otherwise, $a_2\geq 1$. Since $a_1\geq a_2\geq a_3\geq 0$ and $a_1+a_2+a_3=n-3$, we have $a_2\leq \lfloor\frac{n-3}{2}\rfloor$. We consider the following cases:

\vspace*{2mm}

\noindent
${\bf Case\,1.1.}$ $a_2=1$. In this case, we have either $a_3=0$ or $a_3=1$. For $a_3=0\,(a_1=n-4)$, we obtain
\begin{align*}
HSO(S(n-4,1,0))&=(n-4)\,\sqrt{(n-2)^2+1}+\sqrt{10}+\sqrt{3.25}\\[3mm]
&~~~~~~+\sqrt{\left(\frac{n-2}{3}\right)^2+1}+\sqrt{\left(\frac{n-2}{2}\right)^2+1}\\[3mm]
&<(n-4)\,\sqrt{(n-2)^2+1}+4.97+\frac{n-1}{3}+\frac{n-1}{2}\\[3mm]
&<(n-3)\,\sqrt{n^2-2n+2}+\sqrt{n^2-2n+5}+\sqrt{2}
\end{align*}
as $n\geq 5$. The result (\ref{e2}) strictly holds. 

\vspace*{2mm}

For $a_3=1\,(a_1=n-5)$, we have 
\begin{align*}
HSO(S(n-5,1,1))&=(n-5)\,\sqrt{(n-3)^2+1}+2\,\sqrt{10}+\sqrt{\left(\frac{n-3}{3}\right)^2+1}\\[3mm]&~~~~~~+\sqrt{2}+\sqrt{\left(\frac{n-3}{3}\right)^2+1}\\[3mm]
&<(n-5)\,\sqrt{(n-3)^2+1}+7.74+\frac{n-1}{3}+\frac{n-1}{3}\\[3mm]
&<(n-3)\,\sqrt{n^2-2n+2}+\sqrt{n^2-2n+5}+\sqrt{2}
\end{align*}
as $n\geq 6$. The result (\ref{e2}) strictly holds. 

\vspace*{2mm}

\noindent
${\bf Case\,1.2.}$ $a_2\geq 2$. Since $a_1+a_2+a_3=n-3$ and $a_1\geq a_2\geq a_3$, we have $a_3+2\leq a_1+2\leq n-a_2-1$. Moreover, $a_2\leq \lfloor\frac{n-3}{2}\rfloor$ and $a_1\leq n-5$. Thus we have
\begin{align*}
a_1\,\sqrt{(a_1+2)^2+1}+a_3\,\sqrt{(a_3+2)^2+1}&\leq (a_1+a_3)\,\sqrt{(n-a_2-1)^2+1}\\[2mm]
&=(n-a_2-3)\,\sqrt{(n-a_2-1)^2+1}.
\end{align*}
Since $a_2\geq 2$, by Lemma \ref{k2}, we obtain
\begin{align*}
&a_2\,\sqrt{(a_2+2)^2+1}+(n-a_2-3)\,\sqrt{(n-a_2-1)^2+1}\\[3mm]
&~~~~~~~~~~~~~~~~~~~~~~~~~~\leq 2\,\sqrt{17}+(n-5)\,\sqrt{(n-3)^2+1}.
\end{align*}
Using the above results, we obtain
\begin{align*}
&a_1\,\sqrt{(a_1+2)^2+1}+a_2\,\sqrt{(a_2+2)^2+1}+a_3\,\sqrt{(a_3+2)^2+1}\\[2mm]
&~~~~~~~~~\leq a_2\,\sqrt{(a_2+2)^2+1}+(n-a_2-3)\,\sqrt{(n-a_2-1)^2+1}\\[2mm]
&~~~~~~~~~\leq 2\,\sqrt{17}+(n-5)\,\sqrt{(n-3)^2+1}.
\end{align*}
Since $n\geq 7$, we obtain
\begin{align*}
\sqrt{\left(\frac{a_1+2}{a_2+2}\right)^2+1}&\leq \sqrt{\left(\frac{n-3}{4}\right)^2+1}<\frac{n-1}{4},\\[3mm]
\sqrt{\left(\frac{a_2+2}{a_3+2}\right)^2+1}&\leq \sqrt{\left(\frac{n+1}{4}\right)^2+1}<\frac{n-1}{2},\\[3mm]
\sqrt{\left(\frac{a_1+2}{a_3+2}\right)^2+1}&\leq \sqrt{\left(\frac{n-3}{2}\right)^2+1}<\frac{n-1}{2}.    
\end{align*}
Using the above results, we obtain
\begin{align*}
\sqrt{\left(\frac{a_1+2}{a_2+2}\right)^2+1}+\sqrt{\left(\frac{a_2+2}{a_3+2}\right)^2+1}+\sqrt{\left(\frac{a_1+2}{a_3+2}\right)^2+1}<\frac{5\,(n-1)}{4}.
\end{align*}
Using the above results, from (\ref{r0}), we obtain
\begin{align*}
HSO(S(a_1,a_2,a_3))&<\frac{5\,(n-1)}{4}+2\,\sqrt{17}+(n-5)\,\sqrt{(n-3)^2+1}\\[3mm]
   &<(n-3)\,\sqrt{n^2-2n+2}+\sqrt{n^2-2n+5}+\sqrt{2}.
\end{align*}
\vspace*{2mm}
The result (\ref{e2}) strictly holds. 

\vspace*{3mm}

\noindent
${\bf Case\,2.}$ $p\leq n-4$. Since $G$ is unicyclic, in this case, $\Delta\leq n-2$. Thus we have $\sqrt{\Delta^2+1}\leq \sqrt{n^2-4n+5}<\sqrt{n^2-2n+2}$ and $\sqrt{\Delta^2+4}\leq \sqrt{n^2-4n+8}<\sqrt{n^2-2n+2}$,
 as $n\geq 4$. Using the above results with (\ref{e3}) and (\ref{e4}), we obtain
\begin{align*}
HSO(G)&=\sum\limits_{v_iv_j\in E(G)}\,\frac{\sqrt{d^{2}_i+d^{2}_j}}{\min\{d_i,\,d_j\}}\\[3mm]
&=\sum\limits_{v_iv_j\in E(G),\atop d_i\geq d_j=1}\,\frac{\sqrt{d^{2}_i+d^{2}_j}}{d_j}+\sum\limits_{v_iv_j\in E(G),\atop d_i\geq d_j\geq 2}\,\frac{\sqrt{d^{2}_i+d^{2}_j}}{d_j}\\[3mm]
&\leq p\,\sqrt{\Delta^2+1}+\frac{n-p}{2}\,\sqrt{\Delta^2+4}\\[3mm]
&=\frac{n}{2}\,\sqrt{\Delta^2+4}+p\,\left(\sqrt{\Delta^2+1}-\frac{1}{2}\,\sqrt{\Delta^2+4}\right)\\[3mm]
&\leq \frac{n}{2}\,\sqrt{\Delta^2+4}+(n-4)\,\left(\sqrt{\Delta^2+1}-\frac{1}{2}\,\sqrt{\Delta^2+4}\right)\\[3mm]
&=(n-4)\,\sqrt{\Delta^2+1}+2\,\sqrt{\Delta^2+4}\\[3mm]
&<(n-3)\,\sqrt{n^2-2n+2}+\sqrt{n^2-2n+5}+\sqrt{2}.
\end{align*}
The result (\ref{e2}) strictly holds. 

\vspace*{3mm}

\noindent
{\bf Lower Bound:} Let $m$ and $p$ be the number of edges and the number of pendent vertices in $G$. Since $G$ is unicyclic, we have $p\geq 0$ and $m=n$. Using this with Lemma \ref{k1}, we obtain
\begin{align*}
HSO(T)&=\sum\limits_{v_iv_j\in E(T)}\,\frac{\sqrt{d^{2}_i+d^{2}_j}}{\min\{d_i,\,d_j\}}
\\[3mm]&=\sum\limits_{v_iv_j\in E(T),\atop d_i\geq d_j=1}\,\frac{\sqrt{d^{2}_i+d^{2}_j}}{\min\{d_i,\,d_j\}}+\sum\limits_{v_iv_j\in E(T),\atop d_i\geq d_j\geq 2}\,\frac{\sqrt{d^{2}_i+d^{2}_j}}{\min\{d_i,\,d_j\}}\\[3mm]
&\geq p\,\sqrt{5}+(m-p)\,\sqrt{2}=n\,\sqrt{2}+p\,(\sqrt{5}-\sqrt{2})\geq n\,\sqrt{2}.
\end{align*}
Moreover, the equality holds if and only if $p=0$ with $d_i=d_j$ for any edge $v_iv_j\in E(G)$, that is, if and only if $G$ is a regular graph, that is, $G\cong C_n$ as $G$ is unicyclic.
\end{proof}

 \begin{figure}[ht!]
 \begin{center}
 \includegraphics[height=4cm]{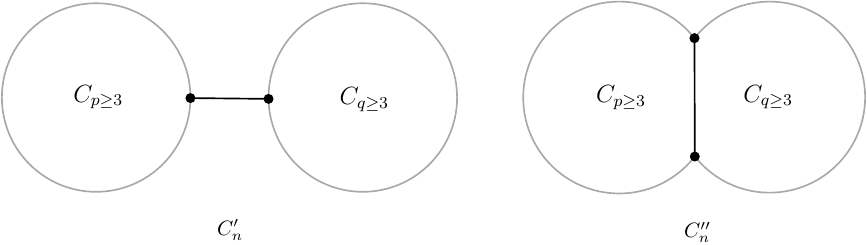}
 \caption{\small \sl Two graphs $C'_n$ and $C''_n$.}
\label{fig3}
\end{center}
\end{figure}

Let $C'_n$ and $C''_n$ denote $n$-vertex bicyclic graphs constructed from cycles $C_p$ ($p\geq 3$) and $C_q$ ($q\geq 3$) as follows: $C'_n$ is obtained by joining $C_p$ and $C_q$ with a single edge (so that $p+q=n$), while $C''_n$ is formed by merging $C_p$ and $C_q$ along a common edge (so that $p+q-2=n$); see Fig. \ref{fig3}. We obtain
 $$HSO(C'_n)=(n-3)\,\sqrt{2}+2\,\sqrt{13}=HSO(C''_n).$$
We now establish a lower bound on $HSO(G)$ for bicyclic graphs in terms of their order $n$, and characterize the corresponding extremal graphs.
\begin{theorem} Let $G$ be a bicyclic graph of order $n$. Then 
\begin{align}
HSO(G)\geq (n-3)\,\sqrt{2}+2\,\sqrt{13}\label{1e2}
\end{align}
with equality if and only if $G\cong C'_n$ or $G\cong C''_n$. 
\end{theorem}

\begin{proof} Let $p\,(\geq 0)$ be the number of pendent vertices in $G$. We consider the following three cases:

\vspace*{2mm}

\noindent
${\bf Case\,1.}$ $p=0$. Since $G$ is bicyclic, the degree sequence of $G$ is $(3,3,\underbrace{2,\ldots,2}_{n-2})$ or $(4,\underbrace{2,\ldots,2}_{n-1})$. First we assume that the degree sequence of $G$ is $(3,3,\underbrace{2,\ldots,2}_{n-2})$. If two vertices of degree $3$ are adjacent, then $G\cong C'_n$ or $G\cong C''_n$. For $G\cong C'_n$, we have
\begin{align*}
HSO(G)&=4\,\sqrt{\frac{13}{4}}+(n-3)\,\sqrt{2}=(n-3)\,\sqrt{2}+2\,\sqrt{13}
\end{align*}
and hence the equality holds in (\ref{1e2}). For $G\cong C''_n$, we have
\begin{align*}
HSO(G)&=4\,\sqrt{\frac{13}{4}}+(n-3)\,\sqrt{2}=(n-3)\,\sqrt{2}+2\,\sqrt{13}
\end{align*}
and hence the equality holds in (\ref{1e2}). Otherwise, two vertices of degree $3$ are not adjacent. Thus we obtain
\begin{align*}
HSO(G)=6\,\sqrt{\frac{13}{4}}+(n-5)\,\sqrt{2}&=(n-5)\,\sqrt{2}+3\,\sqrt{13}\\[2mm]
&>(n-3)\,\sqrt{2}+2\,\sqrt{13}.
\end{align*}
The lower bound in (\ref{1e2}) strictly holds.

\vspace*{3mm}

Next we assume that the degree sequence of $G$ is $(4,\underbrace{2,\ldots,2}_{n-1})$. Thus we obtain
\begin{align*}
HSO(G)&=4\,\sqrt{5}+(n-3)\,\sqrt{2}>(n-3)\,\sqrt{2}+2\,\sqrt{13}.
\end{align*}
The lower bound in (\ref{1e2}) strictly holds.

\vspace*{3mm}

\noindent
${\bf Case\,2.}$ $p=1$. Let $v_r$ and $v_k$ be the maximum degree vertex and the pendent vertex in $G$. Then $d_r=\Delta$ and $d_k=1$. Then we have
  $$2\,(n+1)=\sum\limits^n_{i=1}\,d_i\geq \Delta+2\,(n-2)+1,~\mbox{ that is, }~3\leq \Delta\leq 5.$$ 
Let $v_i\,(\neq v_k)$ be any vertex in $G$. Then $d_i\in \{2,\,3,\,4,\,5\}$.
First we assume that the pendent vertex $v_k$ is adjacent to the vertex $v_{\ell}$ of degree $3$ or more. Then we have
 $$\frac{\sqrt{d^{2}_{\ell}+d^{2}_k}}{d_k}\geq \sqrt{10}.$$
Using the above result with Lemma \ref{k1}, we obtain
\begin{align*}
HSO(G)&=\sum\limits_{v_iv_j\in E(G)}\,\frac{\sqrt{d^{2}_i+d^{2}_j}}{\min\{d_i,\,d_j\}}=\frac{\sqrt{d^{2}_{\ell}+d^{2}_k}}{d_k}+\sum\limits_{v_iv_j\in E(G),\atop d_i\geq d_j\geq 2}\,\frac{\sqrt{d^{2}_i+d^{2}_j}}{d_j}\\[3mm]
&\geq \sqrt{10}+n\,\sqrt{2}>(n-3)\,\sqrt{2}+2\,\sqrt{13}.
\end{align*}
The lower bound in (\ref{1e2}) strictly holds.

\vspace*{2mm}

Next we assume that only the pendent vertex $v_k$ is adjacent to the vertex $v_{\ell}$ of degree $2$. Then we have
\begin{align}
\frac{\sqrt{d^{2}_{\ell}+d^{2}_k}}{d_k}=\sqrt{5}.\label{t1}
\end{align}
Let $E_1=\{v_rv_j\in E(G)|\,v_j\in N_G(v_r)\}$. For $\Delta=5$, the degree sequence of $G$ is $(5,\underbrace{2,\ldots,2}_{n-2},\,1)$. For $\Delta=4$, the degree sequence of $G$ is $(4,3,\underbrace{2,\ldots,2}_{n-3},\,1)$. For $4\leq \Delta\leq 5$, one can easily see that the maximum degree vertex $v_r$ is adjacent to at least two vertices of degree $2$.

\vspace*{2mm}

For $\Delta=3$, the degree sequence of $G$ is $(3,3,3,\underbrace{2,\ldots,2}_{n-4},\,1)$. Since $G$ is bicyclic and $n\geq 10$, one can easily see that there exists a vertex $v_r$ of degree $3$ is adjacent to at least two vertices of degree $2$. For $4\leq \Delta\leq 5$ or $\Delta=3$, using the above result with Lemma \ref{k1}, we obtain
  $$\sum\limits_{v_j:v_rv_j\in E(G)}\,\frac{\sqrt{d^{2}_r+d^{2}_j}}{d_j}\geq \sqrt{13}+(d_r-2)\,\sqrt{2}.$$
Using the above result with (\ref{t1}), we obtain
\begin{align*}
HSO(G)&=\frac{\sqrt{d^{2}_{\ell}+d^{2}_k}}{d_k}+\sum\limits_{v_j:v_rv_j\in E(G)}\,\frac{\sqrt{d^{2}_r+d^{2}_j}}{d_j}+\sum\limits_{v_iv_j\in E(G)\backslash E_1,\atop d_i\geq d_j\geq 2}\,\frac{\sqrt{d^{2}_i+d^{2}_j}}{d_j}\\[3mm]
&\geq \sqrt{5}+\sqrt{13}+(d_r-2)\,\sqrt{2}+(n-d_r)\,\sqrt{2}>(n-3)\,\sqrt{2}+2\,\sqrt{13}.
\end{align*}
The lower bound in (\ref{1e2}) strictly holds.

\vspace*{2mm}

\noindent
${\bf Case\,3.}$ $p\geq 2$. By Lemma \ref{k1}, we obtain
\begin{align*}
HSO(G)&=\sum\limits_{v_iv_j\in E(G),\atop d_i\geq d_j=1}\,\frac{\sqrt{d^{2}_i+d^{2}_j}}{d_j}+\sum\limits_{v_iv_j\in E(G),\atop d_i\geq d_j\geq 2}\,\frac{\sqrt{d^{2}_i+d^{2}_j}}{d_j}\\[3mm]
&\geq p\,\sqrt{5}+(n+1-p)\,\sqrt{2}\\[2mm]
&=(n+1)\,\sqrt{2}+p\,(\sqrt{5}-\sqrt{2})\\[2mm]
&\geq (n+1)\,\sqrt{2}+2\,(\sqrt{5}-\sqrt{2})\\[2mm]
&=2\,\sqrt{5}+(n-1)\,\sqrt{2}>(n-3)\,\sqrt{2}+2\,\sqrt{13}.
\end{align*}
The lower bound in (\ref{1e2}) strictly holds.
\end{proof}
Let $C_{3,3}$ denote the graph obtained by merging one vertex from each of two 3-cycles (triangles), and then attaching $n - 5$ pendent edges to the common (merged) vertex; see Fig. \ref{fig4}. In total, the graph has $n$ vertices. We have
    $$HSO(C_{3,3})=(n-5)\,\sqrt{n^2-2n+2}+2\,\sqrt{n^2-2n+5}+2\,\sqrt{2}.$$
We define $S''_n$ as a connected bicyclic graph of order $n$, constructed from $K_4 - e$ (where $e$ is any edge of $K_4$) by attaching $n-4$ pendent edges to one of its vertices of degree $3$; see Fig. \ref{fig4}. We have
    $$HSO(S''_n)=(n-4)\,\sqrt{n^2-2n+2}+\sqrt{n^2-2n+5}+\frac{1}{3}\,\sqrt{n^2-2n+10}+\sqrt{13}.$$

 \begin{figure}[ht!]
 \begin{center}
 \includegraphics[height=4.5cm]{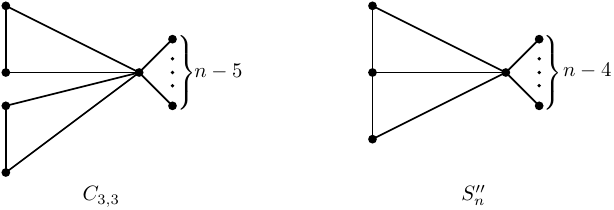}
 \caption{\small \sl Two graphs $C_{3,3}$ and $S''_n$.}
\label{fig4}
\end{center}
\end{figure}
    
We now establish an upper bound of $HSO(G)$ for bicyclic graphs in terms of their order $n$, and characterize the corresponding extremal graphs.
\begin{theorem} Let $G$ be a bicyclic graph of order $n$. Then 
{\small
\begin{align}
HSO(G)&\leq (n-4)\,\sqrt{n^2-2n+2}+\sqrt{n^2-2n+5}+\frac{1}{3}\,\sqrt{n^2-2n+10}+\sqrt{13}\label{pe2}
\end{align}}
with equality if and only if $G\cong S''_n$. 
\end{theorem}

\begin{proof} For $4\leq n\leq 9$, by Sage \cite{SA}, one can easily see that the result holds. Moreover, the equality holds if and only if $G\cong S''_n$. Otherwise, $n\geq 10$. Let $v_1$ be the maximum degree vertex of degree $\Delta$. For any $v_i\in V(G)$, one can easily see that
\begin{align}
\sqrt{d^2_i+1}<d_i+\frac{1}{2\,d_i}~\mbox{ and }~\sqrt{d^2_i+4}<d_i+\frac{2}{d_i}.\label{1pp1}
\end{align}
Let $v_2$ be the second maximum degree vertex of degree $d_2$ in $G$. Let $HSO(e_i)$ be the contribution of an edge $e_i$ to $HSO(G)$. From the definition of the Hyperbolic Sombor index, we have
\begin{align}
HSO(G)=\sum\limits_{e_i\in E(G)}\,HSO(e_i).\label{BG1}
\end{align}
We consider the following five cases:

\vspace*{3mm}

\noindent
${\bf Case\,1.}$ $\Delta=n-1$. Since $G$ is bicyclic, we have $G\cong S''_n$ or $G\cong C'_{3,3}$. For $G\cong S''_n$, we have
$$HSO(G)=(n-4)\,\sqrt{n^2-2n+2}+\sqrt{n^2-2n+5}+\frac{1}{3}\,\sqrt{n^2-2n+10}+\sqrt{13}$$
and hence the equlity holds in (\ref{pe2}). For $G\cong C'_{3,3}$, we have
{\small
\begin{align*}
HSO(G)&=(n-5)\,\sqrt{n^2-2n+2}+2\,\sqrt{n^2-2n+5}+2\,\sqrt{2}\\[3mm]
   &<(n-4)\,\sqrt{n^2-2n+2}+\sqrt{n^2-2n+5}+\frac{1}{3}\,\sqrt{n^2-2n+10}+\sqrt{13}.
\end{align*}}
The upper bound in (\ref{pe2}) strictly holds.

\vspace*{3mm}

\noindent
${\bf Case\,2.}$ $\Delta=n-2$. In this case we have $d_2\leq 4$. 
Let $e_1,\,e_2,\,e_3$ be the edges in $G$ such that $E(G)=\{v_1v_i,\,v_i\in N_G(v_1)\}\cup \{e_1,\,e_2,\,e_3\}$. For $1\leq i\leq 3$ with $e_i=v_jv_k$, we obtain
  $$HSO(e_i)=\frac{\sqrt{d^{2}_j+d^{2}_k}}{\min\{d_j,\,d_k\}}\leq \sqrt{4^2+1}=\sqrt{17}.$$
Thus we have
$$\sum\limits^3_{i=1}\,HSO(e_i)\leq 3\,\sqrt{17}.$$
Since $\Delta=|N_G(v_1)|=n-2$, there are at least two vertices $v_s$ and $v_t$ are of degree $2$ or more, where $v_s,\,v_t\in N_G(v_1)$. Let $v_1v_s=e_4$ and $v_1v_t=e_5$. Thus we have
$$HSO(e_4)\leq \sqrt{\Big(\frac{\Delta}{2}\Big)^2+1}~\mbox{ and }~HSO(e_5)\leq \sqrt{\Big(\frac{\Delta}{2}\Big)^2+1}.$$
Since $n\geq 10$, by Lemma \ref{k1} with (\ref{1pp1}), from (\ref{BG1}), we obtain

\begin{align*}
HSO(G)&=\sum\limits_{e_i\in E(G)}\,HSO(e_i)\\[3mm]
&\leq (\Delta-2)\,\sqrt{\Delta^2+1}+2\,\sqrt{\Big(\frac{\Delta}{2}\Big)^2+1}+3\,\sqrt{17}\\[1mm]
&<(n-4)\,\sqrt{(n-2)^2+1}+\sqrt{(n-2)^2+4}+12.37\\[3mm]
&<(n-4)\,\Big(n-2+\frac{1}{2\,(n-2)}\Big)+\Big(n-2+\frac{2}{n-2}\Big)+12.37\\[3mm]
&=n^2-5n+18.87+\frac{1}{n-2}\\[3mm]
&<(n-4)\,(n-1)+\sqrt{n^2-2n+5}+\frac{1}{3}\,\sqrt{n^2-2n+10}+\sqrt{13}\\[3mm]
&<(n-4)\,\sqrt{n^2-2n+2}+\sqrt{n^2-2n+5}+\frac{1}{3}\,\sqrt{n^2-2n+10}\\[1mm]&~~~~~~+\sqrt{13}.
\end{align*}
The upper bound in (\ref{pe2}) strictly holds.

\vspace*{3mm}

\noindent
${\bf Case\,3.}$ $\Delta=n-3$. In this case we have $d_2\leq 5$. 
Let $e'_1,\,e'_2,\,e'_3,\,e'_4$ be the edges in $G$ such that $E(G)=\{v_1v_i,\,v_i\in N_G(v_1)\}\cup \{e'_1,\,e'_2,\,e'_3,\,e'_4\}$. For $1\leq i\leq 4$ with $e'_i=v_jv_k$, we obtain
  $$HSO(e'_i)=\frac{\sqrt{d^{2}_j+d^{2}_k}}{\min\{d_j,\,d_k\}}\leq \sqrt{5^2+1}=\sqrt{26}.$$
Thus we have
$$\sum\limits^4_{i=1}\,HSO(e'_i)\leq 4\,\sqrt{26}.$$
By (\ref{1pp1}), we obtain
\begin{align*}
\Delta\,\sqrt{\Delta^2+1}=(n-3)\,\sqrt{(n-3)^2+1}&<(n-3)\,\Big(n-3+\frac{1}{2\,(n-3)}\Big)\\
&=n^2-6n+9.5.
\end{align*}
Since $n\geq 10$, by Lemma \ref{k1} with the above results, from (\ref{BG1}), we obtain
{\small
\begin{align*}
HSO(G)&=\sum\limits_{e_i\in E(G)}\,HSO(e_i)\\[1mm]
&\leq \Delta\,\sqrt{\Delta^2+1}+4\,\sqrt{26}\\[1mm]
&<n^2-6n+29.9\\[1mm]
&<(n-4)\,\sqrt{n^2-2n+2}+\sqrt{n^2-2n+5}+\frac{1}{3}\,\sqrt{n^2-2n+10}+\sqrt{13}.
\end{align*}
}
The upper bound in (\ref{pe2}) strictly holds.

\vspace*{3mm}

\noindent
${\bf Case\,4.}$ $\Delta=n-4$. In this case we have $d_2\leq 6$. 
Let $e''_1,\,e''_2,\,e''_3,\,e''_4,\,e''_5$ be the edges in $G$ such that $E(G)=\{v_1v_i,\,v_i\in N_G(v_1)\}\cup \{e''_1,\,e''_2,\,e''_3,\,e''_4,\,e''_5\}$. For $1\leq i\leq 5$ with $e''_i=v_jv_k$, we obtain
  $$HSO(e''_i)=\frac{\sqrt{d^{2}_j+d^{2}_k}}{\min\{d_j,\,d_k\}}\leq \sqrt{6^2+1}=\sqrt{37}.$$
Thus we have
$$\sum\limits^5_{i=1}\,HSO(e''_i)\leq 5\,\sqrt{37}.$$
Since $\Delta=n-4$, by (\ref{1pp1}), we have
\begin{align*}
\Delta\,\sqrt{\Delta^2+1}=(n-4)\,\sqrt{(n-4)^2+1}&<(n-4)\,\Big(n-4+\frac{1}{2\,(n-4)}\Big)\\
  &=n^2-8n+16.5.
\end{align*}
Since $n\geq 10$, by Lemma \ref{k1} with the above results, from (\ref{BG1}), we obtain
\begin{align*}
HSO(G)&=\sum\limits_{e_i\in E(G)}\,HSO(e_i)\\[2mm]
&\leq \Delta\,\sqrt{\Delta^2+1}+5\,\sqrt{37}\\[2mm]
&<n^2-8n+46.92\\[2mm]
&<(n-3)\,(n-1)+\frac{1}{3}\,\sqrt{n^2-2n+10}+\sqrt{13}\\[3mm]
&<(n-4)\,\sqrt{n^2-2n+2}+\sqrt{n^2-2n+5}+\frac{1}{3}\,\sqrt{n^2-2n+10}\\[1mm]&~~~~~~~~~~~~~~~~~~~~~~~~~~~~+\sqrt{13}.
\end{align*}
The upper bound in (\ref{pe2}) strictly holds.

\vspace*{3mm}

\noindent
${\bf Case\,5.}$ $\Delta\leq n-5$. Similarly, as in {\bf Case 4}, we have
\begin{align*}
(n+1)\,\sqrt{\Delta^2+1}\leq (n+1)\,\sqrt{(n-5)^2+1}&<(n+1)\,\Big(n-5+\frac{1}{2\,(n-5)}\Big)\\
  &=n^2-4n-5+\frac{n+1}{2\,(n-5)}\\
  &<(n-1)\,(n-3),
\end{align*}
and hence
\begin{align*}
HSO(G)&=\sum\limits_{e_i\in E(G)}\,HSO(e_i)\\[2mm]
&\leq (n+1)\,\sqrt{\Delta^2+1}\\[2mm]
&<(n-3)\,(n-1)\\
&<(n-4)\,\sqrt{n^2-2n+2}+\sqrt{n^2-2n+5}+\frac{1}{3}\,\sqrt{n^2-2n+10}\\[1mm]&~~~~~~+\sqrt{13}.
\end{align*}
The upper bound in (\ref{pe2}) strictly holds. This completes the proof of the theorem.
\end{proof}

\section{Upper and Lower bounds on the Hyperbolic Sombor index of graphs}\label{section4}
In this section we derive some lower and upper bounds for $HSO(G)$ in terms of the number of edges, maximum degree and minimum degree, and we charcterize the graphs that attains these bounds. 
\begin{theorem} Let $G$ be a graph of order $n$ with $m$ edges and maximum degree $\Delta$, minimum degree $\delta$. Then
  $$\Big(1+\frac{\delta}{\sqrt{\Delta^2+\delta^2}+\Delta}\Big)\,m\leq HSO(G)\leq \left(\frac{\Delta}{\delta}+\sqrt{2}-1\right)\,m$$
with both equalities hold if and only if $G$ is a regular graph.
\end{theorem}

\begin{proof} Let $v_iv_j$ be an edge in $G$ with $d_i\geq d_j$. Since the maximum degree $\Delta$ and the minimum degree $\delta$, we have $1\leq \frac{d_i}{d_j}\leq \frac{\Delta}{\delta}$. One can easily see that
\begin{align}
\sqrt{2}+1\leq \sqrt{\frac{d^2_i}{d^2_j}+1}+\frac{d_i}{d_j}\leq \sqrt{\frac{\Delta^2}{\delta^2}+1}+\frac{\Delta}{\delta}=\frac{\sqrt{\Delta^2+\delta^2}+\Delta}{\delta}.\label{1w1}
\end{align}
Moreover, the above left equality holds if and only if $d_i=d_j$, and the right equality holds if and only if $d_i=\Delta$, $d_j=\delta$.

\vspace*{3mm}

\noindent
{\bf Lower Bound:} Using (\ref{1w1}), we obtain
\begin{align*}
\sqrt{\frac{d^2_i}{d^2_j}+1}-\frac{d_i}{d_j}&=\frac{1}{\sqrt{\frac{d^2_i}{d^2_j}+1}+\frac{d_i}{d_j}}\geq \frac{\delta}{\sqrt{\Delta^2+\delta^2}+\Delta},
\end{align*}
which implies that
   $$\sqrt{\frac{d^2_i}{d^2_j}+1}\geq \frac{d_i}{d_j}+\frac{\delta}{\sqrt{\Delta^2+\delta^2}+\Delta}\geq 1+\frac{\delta}{\sqrt{\Delta^2+\delta^2}+\Delta}$$
with equality if and only if $\Delta=\delta$. 

\vspace*{3mm}

Using this, we obtain
\begin{align*}
HSO(G)=\sum\limits_{v_iv_j\in E(G)}\,\frac{\sqrt{d^{2}_i+d^{2}_j}}{\min\{d_i,\,d_j\}}&\geq \sum\limits_{v_iv_j\in E(G)}\,\Big(1+\frac{\delta}{\sqrt{\Delta^2+\delta^2}+\Delta}\Big)\\[3mm]
&=\Big(1+\frac{\delta}{\sqrt{\Delta^2+\delta^2}+\Delta}\Big)\,m.
\end{align*}
Moreover, the above equality holds if and only if $\Delta=\delta$, that is, if and only if $G$ is a regular graph.

\vspace*{3mm}

\noindent
{\bf Upper Bound:} Using (\ref{1w1}), we obtain
\begin{align*}
\sqrt{\frac{d^2_i}{d^2_j}+1}-\frac{d_i}{d_j}&=\frac{1}{\sqrt{\frac{d^2_i}{d^2_j}+1}+\frac{d_i}{d_j}}\leq \frac{1}{\sqrt{2}+1}=\sqrt{2}-1,
\end{align*}
which implies that
   $$\sqrt{\frac{d^2_i}{d^2_j}+1}\leq \frac{d_i}{d_j}+\sqrt{2}-1\leq \frac{\Delta}{\delta}+\sqrt{2}-1$$
with equality if and only if $\Delta=\delta$. 

\vspace*{3mm}

Using this, we obtain
\begin{align*}
HSO(G)&=\sum\limits_{v_iv_j\in E(G)}\,\frac{\sqrt{d^{2}_i+d^{2}_j}}{\min\{d_i,\,d_j\}}\\[1mm]
&\leq \sum\limits_{v_iv_j\in E(G)}\,\Big(\frac{\Delta}{\delta}+\sqrt{2}-1\Big)=\Big(\frac{\Delta}{\delta}+\sqrt{2}-1\Big)\,m.
\end{align*}
Moreover, the above equality holds if and only if $\Delta=\delta$, that is, if and only if $G$ is a regular graph.
This completes the proof of the theorem.
\end{proof}

\section{Concluding Remarks}\label{section5}

Very recently, Barman et al. \cite{BD1} introduced the Hyperbolic Sombor index of a graph $G$ and established several related mathematical results. However, some of the proofs presented in their work contain inaccuracies. In this paper, we address and correct those errors, and further contribute by deriving new results concerning the Hyperbolic Sombor index for various classes of graphs, including trees, unicyclic graphs, and bicyclic graphs. Moreover, we presented some lower and upper bounds for $HSO(G)$ in terms of the number of edges, maximum degree and minimum degree, and we charcterized the graphs that attains these bounds.

\vspace{2mm}

There remain several unexplored directions in this area, which can be the focus of future research. In particular, we propose the following open problems:

\vspace{2mm}

\noindent
\textbf{Problem 1.} Determine the maximal graphs with respect to the Hyperbolic Sombor index when the graph order $n$ and the number of pendent vertices $p$ are fixed.

\vspace{2mm}

\noindent
\textbf{Problem 2.} Determine the maximal graphs with respect to the Hyperbolic Sombor index when the graph order $n$ and the chromatic number $k$ are fixed.

\vspace{2mm}
From {\bf Remark 4}, we present the following conjecture.

\vspace{2mm}
\noindent
\textbf{Conjecture 1.} For any connected graph $G$, $HSO(G)\leq HSO(S_n)$.







\end{document}